\documentclass{fundam}

\usepackage{url} 
\usepackage[ruled,lined]{algorithm2e}
\usepackage{graphicx,tikz}
\usepackage{mathrsfs,calrsfs}
\usepackage{amsfonts,amssymb}
\newtheorem{observation}{Observation}[section]
\newtheorem{conjecture}{Conjecture}[section]
\newcommand{\diam}{{\rm{diam}}}

\begin{document}

\setcounter{page}{47}
\publyear{22}
\papernumber{2165}
\volume{190}
\issue{1}

\finalVersionForARXIV

\title{Diameter of General Kn\"odel Graphs}


\author{Seyed Reza Musawi\thanks{Address for correspondence: Faculty of Mathematical Sciences,
                              Shahrood University of Technology,  P.O. Box 36199-95161, Shahrood, Iran}
       \\
Faculty of Mathematical Sciences \\
Shahrood University of Technology \\
P.O. Box 36199-95161, Shahrood, Iran\\
r\_musawi@shahroodut.ac.ir
\and Esmaeil Nazari\\
Department of Mathematics \\
Tafresh University \\
Tafresh, Iran\\
nazari.esmaeil@gmail.com}

\maketitle

\runninghead{S.R. Musawi and E. Nazari}{Diameter of General Kn\"odel Graphs}

\begin{abstract}
The Kn\"odel graph $W_{\Delta,n}$ is a $\Delta$-regular bipartite graph on $n\geqslant 2^{\Delta}$ vertices where $n$ is an even integer.   In this paper We obtain some results about the distances of two vertices in the Kn\"odel graphs and by them, we prove that $\diam(W_{\Delta,n})=1+\lceil\frac{n-2}{2^{\Delta}-2}\rceil$, where $\Delta\geqslant2$ and $n\geqslant (2\Delta-5)(2^{\Delta}-2)+4$.\vspace*{-4mm}
\end{abstract}


\section{Introduction}
In this paper, all graphs are simple and finite. A simple and finite graph $G=(V,E)$ consists of two finite sets, $V\ne\emptyset$ is the set of its vertices and $E$ is a set of some two-elements subset of $V$. If $E\ne\emptyset$, then each element of  $E$ is called an edge of $G$. We denoted the edge $\{x,y\}$ by $xy$ and we call $x$ and $y$ as the end points of $xy$. Two vertices are called adjacent if they are the end points of an edge. The set of all vertices adjacent to a vertex $x$ is denoted by $N(x)$.
A graph is bipartite if its vertex set can be partitioned into two subsets so that every edge has one end points in each of them.
 A walk in a  simple graph is a sequence $x_0x_1x_2\cdots x_{\ell}$, whose terms are the vertices of the graph such that each two consecutive vertices are adjacent. We say that the walk $x_0x_1x_2\cdots x_{\ell}$ connects $x_0$ to $x_{\ell}$ and refer to it as $x_0x_{\ell}$-walk and the number $\ell$ is called the length of the walk.
A path in a graph is a walk with distinct vertices in it. Given two vertices $x$ and $y$, the distance between them, denoted by $d(x,y)$, is the length of the shortest $xy$-path.
  The diameter of a graph $G$, $\diam(G)$, is the greatest distance between two vertices of $G$. For more terminology, we refer the reader to \cite{bm}.

Two intrested and exciting concepts in communication networks are gossiping and broadcasting problems.  In broadcasting problems, a person has some informations that have to be communicated to others, while in gossiping problems, each person in the network knows a part of the subject and wants to communicate it to others.
In gossiping, if two people can talk to each other, for example, through a telephone conversation, they will pass all their information to each other. For this purpose, various questions have been raised and examined: A person may not be able to communicate with everyone, Multi-person conversations may occur and etc. A gossiping is complete when everyone knows the complete information.

\medskip
Let $f(n)$ be the minimum number of calls necessary to complete a gossiping among $n$ people where any pair of people may call each other (complete graph), it has been proven by various methods that
$f(1)=0, f(2)=1, f(3)=3, f(n)=2(n-2)$, for $n\geqslant4$.

When the communication graph be a tree, Harary and Schwenk obtained that $f(n)=2n-3$ for $n\geqslant2$, and so for any connected communication graph we have $2n-4\leqslant f(n)\leqslant 2n-3$ for $n\geqslant2$.\cite{hhl}

\medskip
For a graph $G$, the minimum number of time units necessary to complete a gossiping (2-party) is denoted by $T(G)$. If $P_n$ be the path of length $n$, then
$T(P_n)=\left\{
\begin{array}{cc}n-1&\text{for} \;n\; \text{even}\\
n&\text{for} \;n\; \text{odd}
\end{array} \right.$
and $T(G_{m,n})$ is equal to the diameter of $G_{m,n}$ (except $G_{3,3}$), where $G_{m,n}$ is the grid graph.

\medskip
In 1977, Slater raised a new question. What is the minimum number of time units to transfer one person's information to the rest of the group? This simple question became the basis for extensive research into the theory and technology of broadcasting in communication, information and computer networks.
Broadcasting starts from one person and we say that is completed when all people are informed.

Consider a connected graph $G$ and assume that the vertex $u$ is the message initiator. The minimum time required to complete the transmission of information from $u$ is denoted by $b(u)$ and it is called \textit{broadcast time of vertex $u$}.
Easily, we see that $b(u)\geqslant\lceil\log_2n\rceil$.
The broadcast center of a Graph $G$ is the set $BC(G)$ consisting all vertices $u$ such that $b(u)=\min\{b(v):v\in V(G)\}$.

In 1981, Slater et al. proved that the broadcast center of a tree consists of a star with at least two vertices. Specialy, the broadcast center of a star graph contains all of the vertices of it.\cite{sch}
  Also, the broadcast time of a graph G is defined by $b(G)=\max\{b(u):u\in V(G)\}$.
    For the complete graph $K_n$ with $n\geqslant2$ vertices, we have $b(K_n) =\lceil \log_2n\rceil$.

  We define a minimal broadcast graph to be a graph $G$ with $n$ vertices such that $b(G) = \lceil \log_2n\rceil$, but for
every edge $e\in E(G)$, $b(G-e)>b(G)$. For example, the cycle graph $C_4$ is a minimal broadcast graph with 4 vertices.

 The broadcast function $B(n)$ is defined as the minimum number of edges in any minimal broadcast graph on $n$ vertices. A minimum broadcast graph is a minimal broadcast graph on $n$ vertices having $B(n)$ edges. From an applications perspective, minimum broadcast graphs represent the cheapest possible communication networks (having the fewest communication lines) in which broadcasting can be accomplished, from any vertex, as fast as theoretically possible.
    The results of some studies suggest that minimum broadcast graphs are extremely difficult to find.

An important family of graphs in graph theory is the Kn\"odel graph introduced in 1975 by Walter Kn\"odel \cite{k}.
Indeed, he provided a solution to this problem: \textit{There are n people, and each of them knows a part of an event. They want to communicate their information to others through two-person conversations. Each conversation lasts for a certain period of time and all the information of each person is transferred to the other person. How long will it take for everyone to know the whole story?}

The following definition of Kn\"odel graphs is extracted from Kn\"odel's proof  \cite{fp}:

\begin{definition}
Let $n$ be a positive even integer, and $\Delta$ be an integer satisfying $1\leqslant \Delta \leqslant \log_2n$, the \textbf{Kn\"odel graph} $W_{\Delta,n}$ is a bipartite regular simple graph on $n$ vertices with degree $\Delta$. The vertices of $W_{\Delta,n}$ are the pairs $(i,j)$ with $i=1,2$ and $0\leqslant j\leqslant n/2-1$. For every $j$, $0\leqslant j\leqslant n/2-1$, there is an edge between vertex $(1, j)$ and every vertex $(2,(j+2^k-1) \pmod {n/2})$, for $k=0,1,\cdots,\Delta-1$.
We say the vertices $(1, j)$ and $(2,(j+2^k-1) \pmod {n/2})$ are connected through dimension k.
\end{definition}


We will show the set of vertices $\{(1,0),(1,1),\cdots,(1,\frac{n}{2}-1)\}$ by
 $U=\{u_0,u_1,\cdots,u_{\frac{n}{2}-1}\}$ and the set of vertices $\{(2,0),(2,1),\cdots,(2,\frac{n}{2}-1)\}$ by
  $V=\{v_0,v_1,\cdots,v_{\frac{n}{2}-1}\}$. Then two vertices $u_i$ and $v_j$ are adjacent if and only if $j\in \{i+2^0-1,i+2^1-1,\cdots,i+2^{\Delta-1}-1\}$ (or $j-i\in  \{2^0-1,2^1-1,\cdots,2^{\Delta-1}-1\}$). Throughout the paper, $U\cup V$ is the vertex set of $W_{\Delta,n}$ and we will use two notations $s=2^{\Delta-1}-1$ and $\mathscr{M}_{\Delta}= \{2^0-1,2^1-1,\cdots,2^{\Delta-1}-1\}$. All calculations on the indices are taken modulo $n/2$. Hence $u_i=u_{i'}$ and $v_j=v_{j'}$ if and only if $i\equiv i'\pmod {n/2},$ and $j\equiv j' \pmod {n/2}$, respectively.

Each Kn\"odel graph is a Cayley graph \cite{hmp} and so is a vertex-transitive graph.
It is known that when $\Delta\ge2$, the Kn\"odel graphs can be defined as Cayley graphs on the semi-direct product
$\mathbb{Z}_2\ltimes\mathbb{Z}_{\frac{n}{2}}$
 with the multiplicative law $(x,y)(x',y')=(x+x',y+(-1)^xy')$, where $x,x'\in\mathbb{Z}_2$ and $y,y'\in\mathbb{Z}_{\frac{n}{2}}$, and with $S=\{(1,2^i-1):0\leqslant i\leqslant\Delta-1\}$ as the set of generators.

Here are two graph automorphisms that we will use later. If we want to map the vertex $u_i$ to the vertex $u_j$, we use the graph automorphism $\sigma$ such that $\sigma(u_k)=u_{j-i+k}$ and $\sigma(v_k)=v_{j-i+k}$ for $k=0,1,\cdots,\frac{n}{2}-1$. Also, if we want to map the vertex $u_i$ to the vertex $v_j$, we use the graph automorphism $\sigma'$ such that $\sigma'(u_k)=v_{i+j-k}$ and $\sigma'(v_k)=u_{i+j-k}$ for $k=0,1,\cdots,\frac{n}{2}-1$.

The Kn\"odel graph has a symmetric structure and good properties in terms of broadcasting and gossiping in interconnected networks. Among the well-known network architectures, the Kn\"odel graph can be considered a suitable candidate for the problem of information dissemination.

Kn\"odel graphs, hypercubs and recursive circulant graphs are 3 well-known network topologies in gossiping and broadcasting \cite{h}.
The interested reader can see \cite{hkmp,hl,hlpr}, for more information about gossiping and broadcasting.
The Kn\"odel graphs $W_{\Delta,2^{\Delta}}$ are minimal broacast graphs for all $\Delta\ge2$ and they are used as strong competitors for hypercubes \cite{b}.
Although good communication properties of $W_{\Delta,2^{\Delta}}$ or $W_{\Delta-1,2^{\Delta}-2}$ are well known, they were not studied for general Kn\"odel graph $W_{\Delta,n}$ \cite{fr}.
$W_{\Delta,2^{\Delta}}$ may be constructed recursively \cite{ah,bhlp}. For example,  by removing the edges $u_tv_t$, $0\leqslant t\leqslant 15$, in $W_{4,32}$, we have two copies of $W_{3,16}$ (See Figure \ref{fig1}).

\begin{figure}[h]\label{fig1}
\begin{center}
\begin{tikzpicture}[scale=.95]
\foreach \i in {0,1,2,...,15} \foreach \t in {1,3,7}{\draw (\i,2)--({\i+\t-16*int((\i+\t)/16)},0); \node at ({\i},2.3){$u_{\i}$};\node at ({\i},-.3){$v_{\i}$}; }
\foreach \i in {0,1,2,...,15}\draw[red,dashed,very thick] (\i,2)--(\i,0);

\foreach \i in {0,1,2,...,15} \foreach \j in {0,2}{\filldraw[fill=white] (\i,\j) circle(3pt); }
\end{tikzpicture}
\end{center}
\vspace{-5mm}
\caption{
$W_{4,32}$ can be constructed by two copies of $W_{3,16}$
}
\end{figure}

The diameter and the distance of vertices in graphs are important parameters. Previously, a number of authors have presented results on the diameter of the Kn\"odel graphs. In 2000, Fertin et al. in \cite{frssv} proved that $\diam(W_{\Delta,2^{\Delta}})=1+\lceil\frac{\Delta}{2}\rceil$. On the other hand, Gul Bahar Oad in \cite{o},  provided  some results about the number of vertices with a particular distance from a fixed vertex in the some special Kn\"odel graphs and an exact value for diameter of $W_{3,n}$, $\diam(W_{3,8})=3$ and $\diam(W_{3,n})=\lceil\frac{n-2}{6}\rceil+1$ where $n\geqslant 10$.
In order to explore the communication properties of Kn\"odel graph, Harutyunyan and Oad performed extensive simulations. The simulation results showed that the Kn\"odel graph has good communication properties. In particular, it has a small diameter and broadcast time. However, they were not able to find and to prove the closed form formulas for diameter, broadcast time and number of vertices from a particular distance. By some computational methods, they were only able to come up with some hypotheses for the diameter of the following classes of Kn\"odel graphs \cite{ho,o} (See Table \ref{table}).

\begin{table}[htbp]\small
\vspace*{-5mm}
\caption{Some special Kn\"odel graphs and their diameters.}
  \label{table}
  \begin{center}
\begin{tabular}{|c|c|c|}
\hline Kn\"odel Graph&Diameter&Tested Degree\\
\hline $W_{\Delta-1,2^{\Delta}-2}$&$\lceil(\Delta+2)/2\rceil$&$3\leqslant \Delta\leqslant 24$\\
\hline $W_{\Delta-1,2^{\Delta}}$&$\lceil(\Delta+2)/2\rceil$&$5\leqslant \Delta\leqslant 24$\\
\hline $W_{\Delta,2^{\Delta}+2}$&$\lfloor(\Delta+2)/2\rfloor$&$4\leqslant \Delta\leqslant 24$\\
\hline $W_{\Delta,2^{\Delta}+4}$&$\lceil(\Delta+2)/2\rceil$&$5\leqslant \Delta\leqslant 24$\\
\hline $W_{\Delta,2^{\Delta}+2^{\Delta-1}-2}$&$\lceil(\Delta+2)/2\rceil$&$3\leqslant \Delta\leqslant 24$\\
\hline
\end{tabular}
\end{center}
\end{table}

In the same years, Grigoryan and Harutyunyan presented an algorithm to find a short path between any two vertices in the Kn\"odel graphs. Then they proved the following theorem:
\begin{theorem}\label{the12}\emph{\cite[Theorem 6 ]{gh}}
For any $0<\epsilon<1$ there exists some $N(\epsilon)$ such that for all $n\geqslant N(\epsilon)$, $\Delta<\log n-(1+\epsilon)\log\log n$
and $i>\epsilon n$ we have $2\lfloor\frac{i}{2^{\Delta-1}-1}\rfloor+1\leqslant d(u_0,w)\leqslant 2\lfloor\frac{i}{2^{\Delta-1}-1}\rfloor+3$, where $w\in\{u_i,v_i\}$ and $2\lfloor\frac{\lfloor n/4\rfloor}{2^{\Delta-1}-1}\rfloor+1\leqslant \diam(W_{\Delta,n})\leqslant 2\lfloor\frac{\lfloor n/4\rfloor}{2^{\Delta-1}-1}\rfloor+3$.
\end{theorem}

Comparing our results with the results of Grigoryan and Harutyunyan in Theorem \ref{the12}, it seems that our results are better.
The first constraint in the Theorem \ref{the12} is the existence of $\epsilon>0$ on which the relationship between $n$ and $\Delta$ depends, while we have not such limitation. Also, the pairs $(\Delta,n)$ satisfying $n\leqslant (2\Delta-5)(2^{\Delta}-2)+4$ are more than pairs applies the condition
$\Delta<\log n-(1+\epsilon)\log\log n$.
The next difference relates to the inequalities $2\lfloor\frac{i}{2^{\Delta-1}-1}\rfloor+1\leqslant d(u_0,w)\leqslant 2\lfloor\frac{i}{2^{\Delta-1}-1}\rfloor+3$, where $w\in\{u_i,v_i\}$ and $i>\epsilon n$. We know that for $i_0=(\Delta-3)(2^{\Delta-1}-1)-(\Delta-2)$ the second inequality does not hold, but if $i_0<i<\frac{n}{2}-i_0$, then both inequalities hold. Our bound is in terms $\Delta$ instead of $\epsilon$ and $n$, which is always equal or less than $\epsilon n$.
Finally, since for each positive integer $n$, we have $2\lfloor\frac{\lfloor n/4\rfloor}{2^{\Delta-1}-1}\rfloor+1\leqslant \lceil\frac{n-2}{2^{\Delta}-2}\rceil+1\leqslant 2\lfloor\frac{\lfloor n/4\rfloor}{2^{\Delta-1}-1}\rfloor+3$, our result on the diameter of Kn\"odel graphs is a confirmation for the lower and upper bounds obtained in Theorem \ref{the12}.

Another important parameter of a Kn\"odel graph is its domination number, the least number of elements of a dominating set, that is, a set of vertices such that any vertex out of it, is adjacent to some vertex in it.
In \cite{hl}, an upper bound broadcast function is obtained by using minimum dominating sets of some Kn\"odel graphs.  However, there is still not much information about the domination number of Kn\"odel graphs. For more information see \cite{mmn1,mmn2,mmnj,xxyf}.

In Section 2, first, we introduce a correspondence between the set of walks in the Kn\"odel graph $W_{\Delta,n}$ and the set $\{\pm \underset{k=1}{\overset{m}{\Sigma}}(-1)^{k}a_k: a_k\in \mathscr{M}_{\Delta} , m=1,2,3,\cdots \}$, and then by each summation we make a walk, with definite length, from a vertex to another vertex. The shortest known walk between two vertices gives us an upper bound for their distance. These upper bounds allow us to obtain some formulas for calculating the distance between two vertices.
  Finally, in Section 3, using the results obtained in section 2, about the distance between vertices, we prove that $\diam(W_{\Delta,n})=1+\lceil\frac{n-2}{2^{\Delta}-2}\rceil$, where the number of vertices is sufficiently large.

\section{Distances in $W_{\Delta,n}$}

In this section, we focus on the distances between the vertices of $W_{\Delta,n}$. In the beginning, we present an observation. Since the definition of adjacency in the Kn\"odel graphs depends entirely on the powers of 2, to describe the paths, we always deal with powers of 2. For this reason, our observation is a number-theoretic observation.

\begin{observation}\label{obs21} If $0\leqslant a_0<a_1<a_2<\cdots<a_k$ then the equation $\underset{i=0}{\overset{k-1}{\sum}}2^{x_i}=\underset{i=0}{\overset{k}{\sum}}2^{a_i}$, in $x_0, x_1, \cdots, x_{k-1}$, has no solution in integers.
\end{observation}

In the next two lemmas, we show the relation between the set of walks in Kn\"odel graphs and the set of finite summations   $\pm \underset{k=1}{\overset{m}{\Sigma}}(-1)^{k-1}a_k$. Let us emphasize that the two summations $ \underset{k=1}{\overset{m_1}{\Sigma}}b_k$ and $ \underset{k=1}{\overset{m_2}{\Sigma}}c_k$ are equal if and only if $m_1=m_2$ and $b_k=c_k$, where $k=1,2,\cdots, m_1$.

\begin{lemma}\label{lem21}
In a Kn\"odel graph $W_{\Delta,n}$, each walk with the length $m$ gives a unique summation $\pm \underset{k=1}{\overset{m}{\Sigma}}(-1)^{k-1}a_k$, where $a_k\in \mathscr{M}_{\Delta}$ for $k=1,2,\cdots,m$. If $i$ and $j$ be the indices of the starting and ending points of the walk, respectively, then $j-i\equiv \pm \underset{k=1}{\overset{m}{\Sigma}}(-1)^{k-1}a_k \pmod {n/2}$.
\end{lemma}

\begin{proof}
The walks are divided into 4 categories in terms of their starting and ending points. Because of the similarity of methods, We describe one category in full and briefly present the other three cases.
\begin{enumerate}
\item The starting and ending points are in $V$: \smallskip\\
 Consider the walk $v_{j_1}u_{i_1}v_{j_2}u_{i_2}\cdots v_{j_{\ell}}u_{i_{\ell}}v_{j_{{\ell}+1}}$, with the length equal $m=2\ell$. For each $t$, $1\leqslant  t\leqslant \ell$, the vertex $u_{i_t}$ is adjacent with two vertices  $v_{j_t}$ and $v_{j_{t+1}}$. By definition of adjacency in the Kn\"odel graphs, we have $i_t-j_t\equiv -a_{2t-1}$ and $j_{t+1}-i_t\equiv a_{2t} \pmod {n/2}$, where $a_{2t-1},a_{2t}\in \mathscr{M}_{\Delta}$. Now , we have constructed the desired summation, that is, $-a_1+a_2-a_3+a_4-\cdots-a_{m-1}+a_m$ or $-\underset{k=1}{\overset{m}{\Sigma}}(-1)^{k-1}a_k$. Also, we have
  \vspace{-.2cm}
 \[-\underset{k=1}{\overset{m}{\Sigma}}(-1)^{k-1}a_k=-\underset{t=1}{\overset{\ell}{\Sigma}}(a_{2t-1}-a_{2t})\equiv -\underset{t=1}{\overset{\ell}{\Sigma}}(j_{t}-j_{t+1})=j_{\ell+1}-j_{1} \pmod {n/2},\] as desired.
\item  The starting and ending points are in $U$: \smallskip\\
 Consider the walk $u_{i_1}v_{j_1}u_{i_2}v_{j_2}\cdots u_{i_{\ell}}v_{j_{\ell}}u_{i_{{\ell}+1}}$, with the length equal $m=2\ell$. In this case, we set $j_t-i_t\equiv a_{2t-1}$ and $i_{t+1}-j_{t}\equiv -a_{2t} \pmod {n/2}$, where $1\leqslant  t\leqslant \ell$ and $a_{2t-1},a_{2t}\in \mathscr{M}_{\Delta}$. The desired summation is $a_1-a_2+a_3-a_4+\cdots+a_{m-1}-a_m$ or $\underset{k=1}{\overset{m}{\Sigma}}(-1)^{k-1}a_k$ and we have $\underset{k=1}{\overset{m}{\Sigma}}(-1)^{k-1}a_k=i_{\ell+1}-i_{1} \pmod {n/2}$.
\item  The starting point is in $V$ and ending point is in $U$: \smallskip\\
 The walk is $\;v_{j_1}u_{i_1}v_{j_2}u_{i_2}\cdots v_{j_{\ell}}u_{i_{\ell}}, \;$ with the length equal $m=2\ell-1$ and we have \\ $-\underset{k=1}{\overset{m}{\Sigma}}(-1)^{k-1}a_k=i_{\ell}-j_{1} \pmod {n/2}$.
\item  The starting point is in $U$ and ending point is in $V$: \smallskip\\
 The walk is $\;u_{i_1}v_{j_1}u_{i_2}v_{j_2}\cdots u_{i_{\ell}}v_{j_{\ell}}, \;$ with the length equal $m=2\ell-1$ and we have \\ $\underset{k=1}{\overset{m}{\Sigma}}(-1)^{k-1}a_k=j_{\ell}-i_{1} \pmod {n/2}$.
\end{enumerate}

\vspace*{-8mm}
\end{proof}

We note that Lemma \ref{lem21} does not claim the existence of a one-to-one correspondence between walks and summations. For example, in the the Kn\"odel graph $W_{3,8}$, four walks $u_0v_1u_2$, $u_1v_2u_3$, $u_2v_3u_0$ and $u_3v_0u_1$ give the summation $1-3$. In fact, in the Kn\"odel graph $W_{\Delta,n}$, each summation is related to $n/2$ different walks. This property is explaind in the following lemma.

\begin{lemma}\label{lem22}
Suppose that  $j-i \equiv \underset{k=1}{\overset{m}{\Sigma}}(-1)^{k-1}a_k \, \pmod {n/2}$, where $m$ is a positive integer, $i,j\in \{0,1,2,\cdots,n/2-1\}$ and $a_k\in \mathscr{M}_{\Delta}$ for $k=1,2,\cdots,m$. We have:
\begin{enumerate}
\itemsep=0.95pt
\item
If $m$ is an even integer, then there exists a walk between $u_i$ and $u_{j}$ with the length $m$.
\item
If $m$ is an odd integer, then there exists a walk between $u_i$ and $v_j$ with the length $m$.
\end{enumerate}
\end{lemma}

\begin{proof}
By definition of Kn\"odel graphs, we know that if $a\in \mathscr{M}_{\Delta}$ then $v_{i+a}\in N(u_i)$ and $u_{i-a}\in N(v_i)$. By this fact, we produce the walk corresponding to the given summation. \medskip\\
\noindent
If $m$ is even,  we consider the walk $u_{i_0} v_{i_1}u_{i_2}v_{i_3}\cdots v_{i_{m-3}} u_{i_{m-2}} v_{i_{m-1}}u_{i_m}$  and if $m$ is odd, we consider the walk $u_{i_0} v_{i_1}u_{i_2}v_{i_3}\cdots v_{i_{m-2}} u_{i_{m-1}} v_{i_{m}}$.
In each case, we define $i_0=i $ and $i_{k}\equiv i_{k-1}+(-1)^{k-1}a_k  \pmod {n/2}$ for $k=1,2,\cdots,m$.
 Now, we see that \[j_m\equiv i_0+\underset{k=1}{\overset{m}{\Sigma}}(-1)^{k-1}a_k=i+\underset{k=1}{\overset{m}{\Sigma}}(-1)^{k-1}a_k\equiv j\, \pmod {n/2}\] Hence, we have $u_{i_m}=u_j$ in case (1) and $v_{i_m}=v_j$ in case (2), and the proof is completed.
\end{proof}

From now on, using vertex transitivity, we choose $u_0$ as root vertex. In the following lemma, we obtain a symmetry in the distances between the vertices of part $U$.

\begin{lemma}\label{lem23}
In each Kn\"odel graph $W_{\Delta,n}$ we have:
  \vspace{-.1cm}
\[d(u_0,u_i)=d(u_0,u_{\frac{n}{2}-i}), \quad  i=1,2,\cdots,\lfloor\frac{n}{4}\rfloor\]
\end{lemma}
\begin{proof}
For each $i\in\{1,2,\cdots,\lfloor\frac{n}{4}\rfloor\}$, we consider the graph automorphism $\sigma_i$, that maps the vertex $u_j$ to the vertex $u_{j+i}$, where $0\leqslant j\leqslant \frac{n}{2}-1$.  Hence, we have
\[d(u_0,u_{\frac{n}{2}-i})=d(\sigma_i(u_0),\sigma_i(u_{\frac{n}{2}-i}))=d(u_i,u_{\frac{n}{2}})=d(u_i,u_0)=d(u_0,u_i),\]
as desired. Note that $\dfrac{n}{2}\equiv 0\, \pmod {n/2}$ and so $u_{\frac{n}{2}}=n_0$.
\end{proof}

The following lemma and its corollary give a lower bound for the distance between $u_0$ and every vertex in part $U$.

\begin{lemma}\label{lem24}
If $(k-1)s<i\leqslant\lfloor\frac{n}{4}\rfloor$ for some positive integer $k$, then $d(u_0,u_i)\geqslant 2k$.
\end{lemma}
\begin{proof}
Assume that  $d(u_0,u_i)=2r$ and the path $u_{0}v_{i_1}u_{j_1}v_{i_2}u_{j_2}\cdots v_{i_r}u_{j_r}$ has the length $2r=d(u_0,u_i)$, where $u_{j_r}=u_{i}$ and so $ i\equiv j_r \pmod {n/2}$. We show that $r\geqslant k$. On the contrary, assume that $r\leqslant k-1$. By Lemma \ref{lem22}, we have $i_1=a_1 $ , $i_t=\underset{l=1}{\overset{t-1}\sum}(a_l-b_l)+a_t, 2\leqslant t\leqslant r $  and $j_t=\underset{l=1}{\overset{t}\sum}(a_l-b_l), 1\leqslant t\leqslant r $, where $a_l,b_l\in\mathscr{M}_\Delta$. Thus,
 $-s\leqslant a_l-b_l \leqslant s$ and
  \vspace{-.2cm}
   \[-\lfloor\frac{n}{4}\rfloor \leqslant -i<-(k-1)s\leqslant -rs\leqslant j_r \leqslant rs\leqslant (k-1)s<i\leqslant \lfloor\frac{n}{4}\rfloor \]
 Now, we have $0< i-j_r<2\lfloor\frac{n}{4}\rfloor\leqslant \frac{n}{2}$ and so $i-j_r \not\equiv 0  \pmod {n/2}$
  or $i \not\equiv j_r  \pmod {n/2}$, a contradiction. Therefore, $r\geqslant k$ that  implies $r\geqslant k$ and $2r=d(u_0,u_i)\geqslant 2k$.
\end{proof}

\begin{corollary}\label{cor25}
If $1\leqslant i\le\lfloor\frac{n}{4}\rfloor$, then $d(u_0,u_i)\geqslant 2\lceil\frac{i}{s}\rceil$.
\end{corollary}
\begin{proof}
We set $k=\lceil\frac{i}{s}\rceil$. Hence, $k-1<\frac{i}{s} \leqslant k$ and so $(k-1)s<i\leqslant ks$. \medskip\\
 Since $(k-1)s<i\leqslant \min\{ks, \lfloor\frac{n}{4}\rfloor\}$, by Lemma \ref{lem24} we obtain that $d(u_0,u_i)\geqslant 2k=2\lceil\frac{i}{s}\rceil$.
\end{proof}

In the next lemma, we calculate the exact value of the distance between $u_0$ and some special vertices in part $U$.

\begin{lemma}\label{lem26}
If $i=ks\leqslant\lfloor\frac{n}{4}\rfloor$ for some positive integer $k$, then $d(u_0,u_i)=2k=\frac{2i}{s}$.
\end{lemma}
\begin{proof}
If $i=ks\le\lfloor\frac{n}{4}\rfloor$ and $k$ is a positive integer, then $ i \not\equiv 0  \pmod {n/2}$ and there exists a path $u_{0}v_{s}u_{s}v_{2s}u_{2s}\dots v_{ks}u_{ks}$ between $u_0$ and $u_{ks}$ with the length equal $2k$. This implies that $d(u_0,u_i)\leqslant 2k$. On the other hand, by Corollary \ref{cor25}, we know that $d(u_0,u_i)\geqslant 2\lceil\frac{i}{s}\rceil=2k$.
 Finaly, by this two inequality, the equality $d(u_0,u_i)=2k=\frac{2i}{s}$ holds, as desired.
\end{proof}

To continue, we have to express a property of the set $\mathscr{M}_{\Delta}$. Indeed, to find the summations introduced in Lemmas \ref{lem21} and  \ref{lem22}, it is sufficient to establish the following two lemmas.

\begin{lemma}\label{lem27}
If $\Delta\geqslant 3$ and $a$ is an integer with $0\leqslant a\leqslant 2^{\Delta-1}-2$ and $a\ne 2^{\Delta-1}-(\Delta-1)$, then the equation $y_1+y_2+\cdots+y_{\Delta-2}=a$ has a solution in $\mathscr{M}_{\Delta-1}$.
\end{lemma}
\begin{proof}We prove this lemma by induction on $\Delta$. If $\Delta=3$, then we have $0\leqslant a\leqslant 1$. Obviously, the equation $y_1=a$ has a solution in $\mathscr{M}_{2}=\{0,1\}$.\\
Assume that the equation $y_1+y_2+\cdots+y_{\Delta-2}=a$ with $0\leqslant a\leqslant 2^{\Delta-1}-2$ and $a\ne 2^{\Delta-1}-(\Delta-1)$ has a solution in $\mathscr{M}_{\Delta-1}$ for some $\Delta\geqslant3$. We show that the equation

\vspace{-3mm}
 $$   \qquad  \qquad\qquad  \qquad   y_1+y_2+\cdots+y_{\Delta-2}+y_{\Delta-1}=a\qquad\qquad\qquad \qquad\qquad(*)$$

\noindent
with $0\leqslant a\leqslant 2^{\Delta}-2$ and $a\ne 2^{\Delta}-\Delta$ has a solution in $\mathscr{M}_{\Delta}$.\\
For this, we consider four distinct cases: \medskip \\
\textbf{Case 1.} If $0\leqslant a\leqslant 2^{\Delta-1}-2$ and $a\ne 2^{\Delta-1}-(\Delta-1)$, then we set $y_{\Delta-1}=0$ and by the induction hypothesis, the equation $y_1+y_2+\cdots+y_{\Delta-2}=a$ has a solution in $\mathscr{M}_{\Delta}$. \smallskip\\
\textbf{Case 2.} If $a=2^{\Delta-1}-(\Delta-1)$, then $y_i=2^i-1$ for $i=1,2,\cdots,\Delta-2$ and $y_{\Delta-1}=1$ give a solution to the equation $(*)$ in $\mathscr{M}_{\Delta}$. \smallskip\\
\textbf{Case 3.} If $2^{\Delta-1}-1\leqslant a\leqslant 2^{\Delta}-3$ and $a\ne 2^{\Delta}-\Delta$, then we set $y_{\Delta-1}=2^{\Delta-1}-1$ and $a'=a-(2^{\Delta-1}-1)$. We have $0\leqslant a'\leqslant 2^{\Delta-1}-2$ and $a'\ne 2^{\Delta-1}-(\Delta-1)$. Now, by induction hypothesis  the equation $y_1+y_2+\cdots+y_{\Delta-2}=a'$ has a solution in $\mathscr{M}_{\Delta}$. \smallskip\\
\textbf{Case 4.} If $a=2^{\Delta}-2$, then $y_1=y_2=2^{\Delta-1}-1$ and $y_3=\cdots=y_{\Delta-1}=0$ give a solution to the equation $(*)$ in $\mathscr{M}_{\Delta}$.
\end{proof}

Here, we consider a specific case. We will use it to verifying the sharpness of an upper bound introduced in Lemma \ref{lem29}.

\begin{lemma}\label{lem28}
If $\Delta\ge3$ and $a_{\Delta}=2^{\Delta-1}-(\Delta-1)$, then the equation $y_1+y_2+\cdots+y_{\Delta-1}=a_{\Delta}$ has a solution in $\mathscr{M}_{\Delta-1}$, but the equation $y_1+y_2+\cdots+y_{\Delta-2}=a_{\Delta}$ has no solution in $\mathscr{M}_{\Delta-1}$.
\end{lemma}
\begin{proof}
If $a_{\Delta}=2^{\Delta-1}-(\Delta-1)$, then by setting  $y_i=2^i-1$ for $i=1,2,\cdots,\Delta-2$ and $y_{\Delta-1}=1$, we have a solution to the equation $y_1+y_2+\cdots+y_{\Delta-1}=a_{\Delta}$ in $\mathscr{M}_{\Delta-1}$.

\medskip
We prove the second statement by contradiction.
On the contrary, assume that the equation $y_1+y_2+\cdots+y_{\Delta-2}=a_{\Delta}$ has a solution  in $\mathscr{M}_{\Delta-1}$, for some $\Delta\ge3$. Hence,
$\underset{i=1}{\overset{\Delta-2}{\sum}}(2^{x_i}-1)=2^{\Delta-1}-(\Delta-1)$, where $x_i$'s are non-negative integers less than $\Delta-1$.
Now, we have $\underset{i=1}{\overset{\Delta-2}{\sum}}2^{x_i}=2^{\Delta-1}-1=\underset{i=1}{\overset{\Delta-1}{\sum}}2^{i-1}$, a contradiction to Observation \ref{obs21}.
\end{proof}

In the next lemma, we obtain an upper bound for the distance between $u_0$ and some vertices of the part $U$ and the exact distance for the other vertices in $U$. For this, Lemma \ref{lem27} will help us.

\begin{lemma}\label{lem29}
 In a Kn\"odel graph $W_{\Delta,n}$, with $\Delta\geqslant 3$ and $n\geqslant 4(\Delta-3)(2^{\Delta-1}-1)+4$ we have:
 \begin{enumerate}
 \itemsep=0.9pt
\item
  If $0\leqslant i \leqslant (\Delta-3)s$, then $d(u_0,u_i)\leqslant 2(\Delta-2) $.
  \item
  If $(\Delta-3)s+1\leqslant \min\{i,n/2-i\} $, then $d(u_0,u_i)=d(u_0,u_{\frac{n}{2}-i})=2\lceil \frac{\min\{i,n/2-i\}}{s}\rceil$.
  \item
 $d(u_0,u_{\lfloor\frac{n}{4}\rfloor})=2\lceil \frac{1}{s}\lfloor\frac{n}{4}\rfloor\rceil$.
 \end{enumerate}
\end{lemma}

\begin{proof}

\vspace*{-9mm}
 \begin{enumerate}
\item
 If $\Delta=3$, then $i=0$ and $d(u_0,u_i)=0\leqslant 2(3-2)$. Suppose that $\Delta\geqslant 4$. We set $a=\lceil \frac{i}{s}\rceil s-i$ and we have $\lceil \frac{i}{s}\rceil \leqslant\Delta-3$ and $0\leqslant a\leqslant s-1=2^{\Delta-1}-2$.
 We will construct a walk between $u_0$ and $u_i$ with the length equal to $2(\Delta-2)$. For this purpose, we show that the equation $i=\underset{\ell=1}{\overset{\Delta-2}\sum}(a_{\ell}-b_{\ell})$ has a solution in $\mathscr{M}_{\Delta}$. There are two cases for $a$.

 \textbf{(i)} $a\ne a_{\Delta}=2^{\Delta-1}-(\Delta-1)$. In this case, we set $a_{\ell}=s$ for $1\leqslant \ell \leqslant \lceil \frac{i}{s}\rceil$ and $a_{\ell}=0$ for $\lceil \frac{i}{s}\rceil+1\leqslant \ell \leqslant \Delta-2$. Now, we have $\lceil \frac{i}{s}\rceil s-i=\underset{\ell=1}{\overset{\Delta-2}\sum}b_{\ell}$
or $a=\underset{\ell=1}{\overset{\Delta-2}\sum}b_{\ell}$. Hence, by Lemma \ref{lem27}, the equation has a solution in $\mathscr{M}_{\Delta}$.

 \textbf{(ii)} $a=a_{\Delta}=2^{\Delta-1}-(\Delta-1)$. In this case, we set $a_{\ell}=s$ for $1\leqslant \ell \leqslant \lceil \frac{i}{s}\rceil$, $a_{\ell}=0$ for $\lceil \frac{i}{s}\rceil+1\leqslant \ell < \Delta-2$ and $a_{\Delta-2}=1$. Now, we have $\lceil \frac{i}{s}\rceil s+1-i=\underset{\ell=1}{\overset{\Delta-2}\sum}b_{\ell}$
or $a+1=\underset{\ell=1}{\overset{\Delta-2}\sum}b_{\ell}$. Since $a+1\ne a_{\Delta}$ and $0\leqslant a+1\leqslant 2^{\Delta-1}-2$, by Lemma \ref{lem27}, the equation has a solution in $\mathscr{M}_{\Delta}$.
\item
 By symmetry, we assume that $i\leqslant \lfloor\frac{n}{4}\rfloor \leqslant \frac{n}{2}-i$ and by Corollary \ref{cor25} we have $d(u_0,u_i)\geqslant 2\lceil \frac{i}{s}\rceil$. We set $k=\lceil \frac{i}{s}\rceil$, $a=ks-i$ and by introducing a walk between $u_0$ and $u_i$ with the length $2k$, we show that $d(u_0,u_i)\leqslant 2\lceil \frac{i}{s}\rceil$.
 Due to the Lemmas \ref{lem21} and \ref{lem22}, we have to solve the equation $i=\underset{\ell=1}{\overset{k}\sum}(a_{\ell}-b_{\ell})$ in $\mathscr{M}_{\Delta}$.
Now, $k\ge\Delta-2$, $0\leqslant a\leqslant s-1=2^{\Delta-1}-2$ and there are two cases for $a$.

 \noindent
 \textbf{(i)} $a\ne a_{\Delta}$. We set $a_{\ell}=s$ for $1\leqslant \ell \leqslant k$ and  $b_{\ell}=0$ for $\Delta-1 \leqslant \ell \leqslant k$, if $k\geqslant \Delta-1$. Therefore, we have $a=ks-i=\underset{i=1}{\overset{\Delta-2}\sum}y_i$ and by Lemma \ref{lem27}, this equation has a solution in $\mathscr{M}_{\Delta}$.

  \noindent
  \textbf{(ii)} $a=a_{\Delta}$. We set $a_{\ell}=s$ for $1\leqslant \ell < k$, $a_{k}=\frac{s-1}{2}$ and  $b_{\ell}=0$ for $\Delta-1 \leqslant \ell \leqslant k$, if $k\geqslant \Delta-1$. Therefore, we have
$(k-1)s+\frac{s-1}{2}-i=\underset{i=1}{\overset{\Delta-2}\sum}y_i$ or $a-\dfrac{s-1}{2}=\underset{i=1}{\overset{\Delta-2}\sum}y_i$.
Since $a-\frac{s-1}{2}\ne a_{\Delta}\!=$ and $0\leqslant a-\frac{s-1}{2}\leqslant 2^{\Delta-1}-2$, by Lemma \ref{lem27} the equation
has a solution in~$\mathscr{M}_{\Delta}$.
\item
 As a special case of (2), Since $(\Delta-3)s+1\leqslant \lfloor\frac{n}{4}\rfloor \leqslant \frac{n}{2}-\lfloor\frac{n}{4}\rfloor $, we have: \vspace*{-1mm}
\[  d(u_0,u_{\lfloor\frac{n}{4}\rfloor})=2\lceil \dfrac{1}{s}\lfloor\dfrac{n}{4}\rfloor\rceil \]
 \end{enumerate}

 \vspace*{-9mm}
\end{proof}

In the following lemma, we show the sharpness of the upper bound in Lemma \ref{lem29}(1).
\begin{lemma}\label{lem210}
Suppose that $\Delta\ge4$ and $n\geqslant (4\Delta-13)(2^{\Delta-1}-1)+(2\Delta-3)$. \\
If $i_{\Delta}=(\Delta-3)(2^{\Delta-1}-1)-(2^{\Delta-2}-(\Delta-2))$, then $d(u_0,u_{i_{\Delta}})=2(\Delta-2)$.
\end{lemma}

\begin{proof}
Since $i_{\Delta}\le(\Delta-3)s$, by Part (1) of Lemma \ref{lem29}, we know that $d(u_0,u_{i_{\Delta}})\le2(\Delta-2)$. We show that the equation $i_{\Delta}=\underset{\ell=1}{\overset{\Delta-3}\sum}(a_{\ell}-b_{\ell})$ has no solution in $\mathscr{M}_{\Delta}$. On the contrary, suppose that the equation \vspace*{-3mm}
\[(\Delta-3)s-(2^{\Delta-2}-(\Delta-2))=\underset{\ell=1}{\overset{\Delta-3}\sum}(a_{\ell}-b_{\ell}) \vspace*{-2mm}\]
 has a solution in $\mathscr{M}_{\Delta}$.
Therefore, the inequalities  \vspace*{-1mm}
\[\underset{\ell=1}{\overset{\Delta-3}\sum}a_{\ell}\geqslant (\Delta-3)s-(2^{\Delta-2}-(\Delta-2))>(\Delta-4)s+\frac{s-1}{2}  \vspace*{-1mm} \]
 implies that $a_{\ell}=s$ for all $\ell=1,2,\cdots, \Delta-3$ and so  \vspace*{-1mm}
\[\underset{\ell=1}{\overset{\Delta-3}\sum}b_{\ell}= 2^{\Delta-2}-(\Delta-2)=i_{\Delta-1}.  \vspace*{-2mm}\]
Since $i_{\Delta-1}< 2^{\Delta-2}-1$ we obtain $b_{\ell}\in\mathscr{M}_{\Delta-2}$, $\ell=1,2,\cdots,\Delta-3$, that is, the equation  $\underset{\ell=1}{\overset{\Delta-3}\sum}b_{\ell}=i_{\Delta-1}$ has a solution in $\mathscr{M}_{\Delta-2}$, a contradiction with Lemma \ref{lem28}.
Hence, we have $d(u_0,u_{i_{\Delta}})>2(\Delta-3)$ and by $d(u_0,u_{i_{\Delta}})\le2(\Delta-2)$, we deduce that $d(u_0,u_{i_{\Delta}})=2(\Delta-2)$.
\end{proof}

\begin{example}
If $\Delta=4$ and $n\ge26$, then we have $5\leqslant (\Delta-3)s$ and the path $u_0v_7u_6v_6u_5$, in Figure~2 is a $u_0u_5$-path with the length $4$.
Since $u_0$ and $u_5$ have no common neighbors, $d(u_0,u_5)>2$ and so $d(u_0,u_5)=4=2(\Delta-2)$. Thus, the upper bound in Lemma \ref{lem29}(1) is sharp.
\end{example}
\begin{figure}[h]\label{fig2}
\begin{center}
\begin{tikzpicture}[scale=1]
\foreach \i in {0,1,2,...,12} \foreach \t in {0,1,3,7}{\draw (\i,2)--({\i+\t-13*int((\i+\t)/13)},0); \node at ({\i},2.3){$u_{\i}$};\node at ({\i},-.3){$v_{\i}$}; }
\draw[red,very thick] (0,2)--(7,0)--(6,2)--(6,0)--(5,2);
\foreach \i in {0,1,2,...,12} \foreach \j in {0,2}{\filldraw[fill=white] (\i,\j) circle(3pt); }
\end{tikzpicture}
\end{center}
\vspace{-5mm}
\caption{ $W_{4,26}$ and a shortest $u_0u_5$-path.}\vspace*{-2mm}
\end{figure}

\begin{corollary}\label{cor212}
In a Kn\"odel graph $W_{\Delta,n}$, with $\Delta\geqslant 3$ and $n\geqslant 4(\Delta-3)(2^{\Delta-1}-1)+4$, if $(\Delta-3)s+1\leqslant i \leqslant j \leqslant \lfloor\frac{n}{4}\rfloor$, then $d(u_0,u_i)\leqslant d(u_0,u_j)\leqslant d(u_0,u_{\lfloor\frac{n}{4}\rfloor})=2\lceil \frac{1}{s}\lfloor\frac{n}{4}\rfloor\rceil$.
\end{corollary}
\begin{proof}
We know that the ceil function is an increasing function and so Lemma \ref{lem29} easily concludes the results.
\end{proof}

The following theorem gives us the maximum distance between the vertices in the part $U$. This value is a candidate for the diameter of $W_{\Delta,n}$.

\begin{corollary}\label{cor213}
  In a Kn\"odel graph $W_{\Delta,n}$, if $\Delta\geqslant 3$ and $n\geqslant 4(\Delta-3)(2^{\Delta-1}-1)+4$, then $\underset{i}{\max}\, d(u_0,u_i)= d(u_0,u_{\lfloor\frac{n}{4}\rfloor})=2\lceil \frac{1}{s}\lfloor\frac{n}{4}\rfloor\rceil$.
  \end{corollary}
\begin{proof}
There are three distinct cases for $i$'s: \medskip\\
\textbf{Case 1:} $i=0,1,\cdots,(\Delta-3)s$. We have $\Delta -2 \leqslant \lceil \frac{(\Delta-3)s+1}{s}\rceil \leqslant \lceil \frac{1}{s}\lfloor\frac{n}{4}\rfloor\rceil$ and by Lemma \ref{lem29}(i),
\vspace{-.2cm}
\[d(u_0,u_i) \leqslant 2(\Delta-2) \leqslant  d(u_0,u_{\lfloor\frac{n}{4}\rfloor})=2\lceil \frac{1}{s}\lfloor\frac{n}{4}\rfloor\rceil.\]
\textbf{Case 2:} $i=(\Delta-3)s+1,\cdots, \lfloor\frac{n}{4}\rfloor$. By Corollary \ref{cor212}, we have
\vspace{-.2cm}
\[d(u_0,u_i) \leqslant  d(u_0,u_{\lfloor\frac{n}{4}\rfloor})=2\lceil \frac{1}{s}\lfloor\frac{n}{4}\rfloor\rceil.\]
\textbf{Case 3:} $i=\lfloor\frac{n}{4}\rfloor+1,\cdots,\frac{n}{2}-1$. We have $1\leqslant \frac{n}{2}-i \leqslant \lfloor\frac{n}{4}\rfloor $. Now, using Lemma \ref{lem23} and prevoius parts, we conclude that
\vspace{-.3cm}
 \[d(u_0,u_i)= d(u_0,u_{\frac{n}{2}-i})
 \leqslant  d(u_0,u_{\lfloor\frac{n}{4}\rfloor})=2\lceil \frac{1}{s}\lfloor\frac{n}{4}\rfloor\rceil.\]
\noindent
 Hence, for each $i$ we have $d(u_0,u_i)
 \leqslant  d(u_0,u_{\lfloor\frac{n}{4}\rfloor})=2\lceil \frac{1}{s}\lfloor\frac{n}{4}\rfloor\rceil$ which implies that:\\
\vspace{-.2cm}

\hspace{4cm}
$\underset{i}{\max}\, d(u_0,u_i)= d(u_0,u_{\lfloor\frac{n}{4}\rfloor})=2\lceil \dfrac{1}{s}\lfloor\dfrac{n}{4}\rfloor\rceil$
\end{proof}

We got some useful results on distances between the vertices in part $U$. Let's move on to the part $V$. First, we note that by transitivity of Kn\"odel graphs, we have $d(u_i,u_j)=d(v_i,v_j)$ for all $i$ and $j$. Then, we take advantage of the fact that the distance of $u_0$ and a vertex  $v_j\in V$ is related to the distances of $u_0$ and the neighbors of $v_j$ in part $U$. After proving the following proposition, we will have the second and last candidate for the diameter of $W_{\Delta,n}$.

  \begin{proposition}\label{prop214}
In the Kn\"odel graph $W_{\Delta,n}$ we have :
 \begin{enumerate}
 \itemsep=-0.95pt
 \item
 If $x$ and $y$ be two adjacent vertices, then we have $|d(u_0,x)-d(u_0,y)|=1$.
\item
 For each $j=0,1,\cdots, \frac{n}{2}-1$ there is an $i=0,1,\cdots, \frac{n}{2}-1$ such that $d(u_0,u_i)=d(u_0,v_j)-1$.
\item
 For each $j=0,1,\cdots, \frac{n}{2}-1$ we have $d(u_0,v_j)=1+\min\{d(u_0,u_i): u_i\in N(v_j)  \}$.
 \item
$\underset{j}{\max}\, d(u_0,v_j)\leqslant 1+2\lceil \dfrac{1}{s}\lfloor\dfrac{n}{4}\rfloor\rceil$.
 \end{enumerate}
\end{proposition}

\begin{proof}

\vspace*{-7mm}
 \begin{enumerate}
 \itemsep=0.95pt
 \item
  If $u_0x_1x_2\cdots x_kx$ be a path between $u_0$ and $x$
, then we can find the walk $u_0x_1x_2\cdots x_kxy$ between $u_0$ and $y$. This shows that $d(u_0,y)\leqslant 1+d(u_0,x)$ and  similarly, we have $d(u_0,x)\leqslant 1+d(u_0,y)$. These two inequalities confirm that $|d(u_0,x)-d(u_0,x)|=1$. Note that $d(u_0,x)$ and $d(u_0,y)$ are unequal in terms of parity.
 \item
 If $d(u_0,v_j)=1$, then we set $i=0$ and we have $d(u_0,u_0)=d(u_0,v_j)-1$. If $d(u_0,v_j)=2r+1$ and $r
\geqslant 1$, then there is a path $u_0v_{j_1}u_{i_1}\cdots v_{j_r}u_{i_r}v_j$ between $u_0$ and $v_j$ with the length $2r+1$. Now, we set $i=i_r$ and so $2r=d(u_0,u_{i_r})=d(u_0,v_j)-1$, as desired.
 \item
 We set $2r=\min\{d(u_0,u_i): u_i\in N(v_j)  \}$. Hence $v_j$ has an adjacent $u_i$ such that $d(u_0,u_i)=2r$ and so by (i) we have $d(u_0,v_j)\in\{2r-1, 2r+1\}$. We have to show that $d(u_0,v_j)=2r+1$. On the contrary, assume that $d(u_0,v_j)=2r-1$. By (ii), $v_j$ has to have an adjacent $u_i$ such that $d(u_0,u_i)=2r-2$, a contradiction by minimality of $2r$.
Therefore, $d(u_0,v_j)=2r+1$ and  the result is obtained.
 \item
  It is obvious that for each $j$, we have $\min\{d(u_0,u_i): u_i\in N(v_j)  \}\leqslant \underset{i}{\max}\, d(u_0,u_i)$ and $\underset{j}{\max}\, d(u_0,v_j)=d(u_0,v_{\ell})$ for some $\ell$.
Now, by (iii) we have $\underset{j}{\max}\, d(u_0,v_j)=1+\min\{d(u_0,u_i): u_i\in N(v_{\ell})  \}\leqslant 1+\underset{i}{\max}\, d(u_0,u_i)$ and finally, by Corollary \ref{cor213} we have $\underset{j}{\max}\, d(u_0,v_j)\leqslant 1+2\lceil \frac{1}{s}\lfloor\frac{n}{4}\rfloor\rceil$ as desired.
 \end{enumerate}

 \vspace*{-8mm}
\end{proof}

Corollary \ref{cor213} gives the maximum value of distances between two vertices in the same part, but Proposition \ref{prop214} gives an upper bound for distances between two vertices in distinct parts. We show that these upper bounds are sharp.
\begin{lemma}\label{lem215}
 In a Kn\"odel graph $W_{\Delta,n}$, with $\Delta\geqslant 3$ and $(4k-2)s+4 \leqslant n \leqslant 4ks+2$ where $k\geqslant \Delta-2$, we have:
 \begin{enumerate}
 \itemsep=0.95pt
 \item
 $d(u_0,u_i)=2k=2\lceil \frac{1}{s}\lfloor\frac{n}{4}\rfloor\rceil$ for $(k-1)s+1\leqslant i \leqslant \frac{n}{2}-(k-1)s-1$.
 \item
 $d(u_0,v_j)=1+2k=1+2\lceil \frac{1}{s}\lfloor\frac{n}{4}\rfloor\rceil$ for $ks+1\leqslant j \leqslant \frac{n}{2}-(k-1)s-1$.
\end{enumerate}
\end{lemma}
\eject

\begin{proof}
Since $n\geqslant 2(2\Delta-5)s+4\geqslant 4(\Delta-3)s+4$, using Corollary \ref{cor213} we have $\underset{i}{\max}\,\, d(u_0,u_i)=2\lceil \frac{1}{s}\lfloor\frac{n}{4}\rfloor\rceil$.
 \begin{enumerate}
 \item
 In the first case, we have $ks+1-\frac{s}{2}\leqslant \frac{n}{4}\leqslant ks+\frac{1}{2}$ and $(k-1)s<ks+1-\frac{s+1}{2}\leqslant \lfloor\frac{n}{4}\rfloor\leqslant ks$. This inequalities imply  that $\lceil \frac{1}{s}\lfloor\frac{n}{4}\rfloor\rceil=k$ and we have $\underset{i}{\max}\,\, d(u_0,u_i)=2k=d(u_0,u_{\lfloor\frac{n}{4}\rfloor})$.
Since $(k-1)s+1\leqslant \min\{i,\frac{n}{2}-i\}$, by Lemma \ref{lem29} we have \[d(u_0,u_i)=d(u_0,u_{\frac{n}{2}-i})=2\lceil \frac{\min\{i,n/2-i\}}{s}\rceil\geqslant 2\lceil \frac{(k-1)s+1}{s}\rceil=2k\]and so by maximality of $d(u_0,u_{\lfloor\frac{n}{4}\rfloor})=2k$,
for $(k-1)s+1\leqslant i \leqslant \frac{n}{2}-(k-1)s-1$,
 we have $d(u_0,u_i)=2k$.
\item
 Now, we consider the vertex $v_j$, where $ks+1\leqslant j \leqslant \frac{n}{2}-(k-1)s-1$ and compute $d(u_0,v_j)$.
From Proposition \ref{prop214} we know that $d(u_0,v_j)=1+\min\{d(u_0,u_i):u_i\in N(v_j)\}$. Assume that $u_i\in N(v_j)$, We have
$i\equiv j-b  \pmod {n/2}$ for some $b\in\mathscr{M}_{\Delta}$. On the other hand, from  $1\leqslant j-s\leqslant j-b \leqslant \frac{n}{2}$
 we deduce that $i=j-b$.
Therefore, $ j-s\leqslant i \leqslant j$ and so $(k-1)s+1\leqslant i \leqslant \frac{n}{2}-(k-1)s-1$. Now, by the previous part, we have $d(u_0,u_i)=2k$, that is, $\{d(u_0,u_i):u_i\in N(v_j)\}=\{2k\}$. Therefore, $d(u_0,v_j)=2k+1$ and the proof is completed.
 \end{enumerate}\vspace*{-4mm}
\end{proof}

 \section{ Main result}
 In this section, we give the exact value of the diameter of some Kn\"odel graphs with sufficiently large order respect to the degree of their vertices.

\begin{lemma}\label{lem31}
 In a Kn\"odel graph $W_{\Delta,n}$, with $\Delta\geqslant 3$ we have:
  \begin{enumerate}
\item
 If $(4k-2)s+4 \leqslant n \leqslant 4ks+2$ and $k\geqslant \Delta-2$, then
  $\diam(W_{\Delta,n})=2k+1$.
 \item
  If $4ks+4 \leqslant n \leqslant (4k+2)s+2$ and $k\geqslant \Delta-2$, then
  $\diam(W_{\Delta,n})=2k+2$.
 \end{enumerate}
  \end{lemma}

\begin{proof}

\vspace*{-7mm}
 \begin{enumerate}
\item
 In this case, we consider the vertex $v_{\lfloor\frac{n+2s}{4}\rfloor}$
and compute $d(u_0,v_{\lfloor\frac{n+2s}{4}\rfloor})$.
Since $(4k-2)s+4 \leqslant n$, we have $4ks+4 \leqslant n+2s
\leqslant 2n-4(k-1)s-4$ and so $ks+1\le\frac{n+2s}{4}\leqslant \frac{n}{2}-(k-1)s-1$. By Lemma \ref{lem215} we have $d(u_0,v_{\lfloor\frac{n+2s}{4}\rfloor})=2k+1=
2\lceil \frac{1}{s}\lfloor\frac{n}{4}\rfloor\rceil+1=\underset{i}{\max}\, d(u_0,u_i)+1$. Now by Corollary \ref{cor213} and Proposition \ref{prop214} we deduce that $\diam(W_{\Delta,n})=d(u_0,v_{\lfloor\frac{n+2s}{4}\rfloor})=2k+1$.
\item
 We have $ks+1\leqslant \lfloor\frac{n}{4}\rfloor\leqslant \frac{n}{4}\leqslant ks+\frac{s+1}{2}$ and $k+\frac{1}{s}\leqslant \frac{1}{s}\lfloor\frac{n}{4}\rfloor\leqslant k+\frac{s+1}{2s}
\leqslant k+1$. This implies that $\lceil \frac{1}{s}\lfloor\frac{n}{4}\rfloor\rceil=k+1$ and so $\underset{i}{\max}\,\, d(u_0,u_i)=2k+2=(u_0,u_{\lfloor\frac{n}{4}\rfloor})$.
Hence, by Proposition \ref{prop214}, we have $d(u_0,v_j)=2k+1$ for some $j$.
We have to prove that $d(u_0,v_j)\leqslant 2k+1$ for $j=0,1,\cdots ,\frac{n}{2}-1$. On the contrary, assume that there exists $j$ such that $d(u_0,v_j)=2k+3$. By Proposition \ref{prop214} we have $\min\, \{d(u_0,u_i):u_i\in N(v_j)\}=  2k+2$ and so
\mbox{$\{d(u_0,u_i): u_i\in N(v_j)\}=$} $\{2k+2\}$, by maximality of $2k+2$. Therefore, $d(u_0,u_j)=d(u_0,u_{j-s})=2k+2$. Now, we claim that $ks+1\leqslant \min\{ j , \frac{n}{2}-j\}$. Otherwise, $\min\{ j , \frac{n}{2}-j \}\leqslant ks$ and by Lemma  \ref{lem29}(i), we have $2k+2=d(u_0,u_j)\leqslant 2(\Delta-2)\leqslant 2k$, a contradiction, or by Lemma \ref{lem29}(ii), we have $2k+2=d(u_0,u_j)=2\lceil \frac{\min\{j,n/2-j\}}{s}\rceil \leqslant 2\lceil \frac{ks}{s}\rceil=2k$, a contradiction.
Since $ks+1\leqslant \min\{ j , \frac{n}{2}-j\}$ we have $ks+1\leqslant j$ and $ks+1\leqslant \frac{n}{2}-j\leqslant 2ks+s+1-j$ or $j\leqslant (k+1)s$. Therefore, we obtain that $ks+1\leqslant j\leqslant (k+1)s$. Now, we have $(\Delta-3)s+1\leqslant (k-1)s+1\leqslant j-s\leqslant ks\leqslant \lfloor\frac{n}{4}\rfloor$ and by Lemma \ref{lem29}(ii), we have $d(u_0,u_{j-s})=2\lceil\frac{j-s}{s}\rceil=2k$, which is a contradiction. Finally, we have $\underset{i}{\max}\,\, d(u_0,v_j)=2k+1$ and so $\diam(W_{\Delta,n})=d(u_0,u_{\lfloor\frac{n}{4}\rfloor})=2k+2$.
 \end{enumerate}

 \vspace*{-9mm}
\end{proof}

Due to the proof of the above lemma, we conclude that:
\begin{corollary}\label{cor32}
 In a Kn\"odel graph $W_{\Delta,n}$, with $\Delta\geqslant 3$ and $n\geqslant (2\Delta-5)(2^{\Delta}-2)+4$, we have
  $\diam(W_{\Delta,n})=\max\,\{d(u_0,u_{\lfloor\frac{n}{4}\rfloor}),d(u_0,v_{\lfloor\frac{n+2s}{4}\rfloor}) \}$.
  \end{corollary}

  We can now state the main purpose of the article.

\begin{theorem}\label{the33}
 In a Kn\"odel graph $W_{\Delta,n}$, with $\Delta\geqslant 3$ and $n\geqslant (2\Delta-5)(2^{\Delta}-2)+4$, we have
  $\diam(W_{\Delta,n})=1+\lceil\frac{n-2}{2^{\Delta}-2}\rceil$.
  \end{theorem}
\begin{proof}
We consider the even integer $n+2s-4$ and by division algorithm we have
$n+2s-4=4ks+2r$ or $n=(4k-2)s+2r+4$, where $k$ and $r$ are integers and $0\leqslant r\leqslant 2s-1$. From now on, we distinguish the following two cases. \medskip
\\
\textbf{Case 1:}
If $0\leqslant r \leqslant s-1$, then $(4k-2)s+4 \leqslant n \leqslant 4ks+2$, that is, $2k-1+\frac{1}{s}\leqslant \frac{n-2}{2s}\leqslant 2k$ and $\lceil\frac{n-2}{2^{\Delta}-2}\rceil=2k$. Now, by Lemma \ref{lem31}(i) we conclude that $\diam(W_{\Delta,n})=2k+1=1+\lceil\frac{n-2}{2^{\Delta}-2}\rceil$. \medskip
\\
\textbf{Case 2:} If $s\leqslant r \leqslant 2s-1$, then $4ks+4 \leqslant n \leqslant (4k+2)s+2$, that is, $2k+\frac{1}{s}\leqslant \frac{n-2}{2s}\leqslant 2k+1$ and $\lceil\frac{n-2}{2^{\Delta}-2}\rceil=2k+1$. Now, by Lemma \ref{lem31}(ii) we conclude that $\diam(W_{\Delta,n})=2k+2=1+\lceil\frac{n-2}{2^{\Delta}-2}\rceil$. This two cases complete the proof.
\end{proof}

\section{Conclusion}
In this article, we discussed about distance and diameter in Kn\"odel graphs $W_{\Delta,n}$, two important concepts in graph theory and communication networks.
We obtained some exact formulas for diameter and the distance between two vertices of Kn\"odel graph for large enough $n$. For smaller values of $n$, we conjecture the inequalities \[1+\lceil\dfrac{n-2}{2^{\Delta}-2}\rceil\leqslant \diam(W_{\Delta,n})\leqslant \lceil\dfrac{\Delta}{2}\rceil+\lceil\dfrac{n-2}{2^{\Delta}-2}\rceil\]
 that remains to be proved.
\begin{conjecture}
If $\Delta\geqslant2$ and $n\geqslant2^{\Delta}$ be an even integer, then $\diam(W_{\Delta,n)}\leqslant\lceil\dfrac{\Delta}{2}\rceil+\lceil\dfrac{n-2}{2^{\Delta}-2}\rceil$.
\end{conjecture}


\end{document}